\documentclass[12pt]{amsart}
\usepackage[active]{srcltx}
\usepackage{latexsym}
\usepackage{amssymb}
\usepackage{verbatim}

\usepackage{amsmath}

\newcommand{\ignore}[1]{}

\newcommand{\hide}[1]{}

\newtheorem{dummy}{Dummy}

\newtheorem{lemma}[dummy]{Lemma}

\newtheorem{theorem}[dummy]{Theorem}

\newtheorem{prop}[dummy]{Proposition}

\theoremstyle{definition}

\newtheorem*{question}{Question}

\theoremstyle{remark}

\newtheorem*{rem*}{Remark to ourselves}

\begin{document}

\bibliographystyle{amsalpha}

\title[On the nilpotency class of some finite groups]{On the nilpotency class of finite groups with a Frobenius group of automorphisms}
 
\author{Valentina Iusa}
\address{School of Mathematics and Physics \\
University of Lincoln \\
Brayford Pool
Lincoln, LN6 7TS\\
United Kingdom}
\email{vIusa@lincoln.ac.uk}

 
\date{June 11, 2018 
}
 

\begin{abstract}
Suppose that a metacyclic Frobenius group $FH$, with kernel $F$ and complement $H$, acts by automorphisms on a finite group $G$, in such a way that $C_G(F)$ is trivial and $C_G(H)$ is nilpotent. It is known that $G$ is nilpotent and its nilpotency class can be bounded in terms of $|H|$ and the nilpotency class of $C_G(H)$. Until now, it was not clear whether the bound could be made independent of the order of $H$. In this article, we construct a family $\mathfrak{G}$ of finite nilpotent groups, of unbounded nilpotency class. Each group in $\mathfrak{G}$ admits a metacyclic Frobenius group of automorphisms such that the centralizer of the kernel is trivial and the centralizer of the complement is abelian. This shows that the dependence of the bound on the order of $H$ is essential.

\end{abstract}

\maketitle

\section{Introduction}
Let $A$ be a group acting on a group $G$. We denote by $C_G(A)$ the centralizer of $A$ in $G$, namely the fixed-point subgroup of $A$. Experience shows that in many cases the properties of $G$ are influenced by those of $C_G(A)$. After Mazurov's problem $17.72$ in Kourovka Notebook \cite{mazurov2014unsolved}, special attention was given to the case where a Frobenius group acts by automorphisms on a finite group.

We recall that a finite Frobenius group $FH$, with kernel $F$ and complement $H$, can be characterized as a semidirect product of a normal subgroup $F$ by $H$ such that $C_F(h)=1$ for every non trivial element $h$ of $H$. The structure of Frobenius groups is well known. In particular, by Thompson's theorem \cite{thompson1959finite} the kernel is nilpotent and, by Higman's theorem \cite{Higman01071957}, its nilpotency class is bounded in terms of the least prime divisor of the order of $H$. The explicit upper bound is due to Kreknin and Kostrikin (see \cite{Kreknin,Kreknin-Kostrikin}).

Suppose that a Frobenius group $FH$ acts on a finite group $G$ in such a way that $C_G(F)$ is trivial. By Belyaev and Hartley's theorem \cite{Belyaev1996}, the group $G$ is soluble. Moreover, by Khukhro--Makarenko--Shumyatsky's theorem \cite[Theorem~2.7]{khukhro2014frobenius} if $C_G(H)$ is nilpotent, then $G$ is nilpotent. In the same article, the three authors proved that if in addition $F$ is cyclic, then the nilpotency class of $G$ can be bounded in terms of the order of $H$ and the nilpotency class of $C_G(H)$. We remark that this result can be seen as part of the more general goal of expressing the properties of $G$ in terms of the corresponding properties of $C_G(H)$, possibly depending also on the order of $H$ (see also \cite{makarenko2010frobenius,SHUMYATSKY2011482,Khukhro2011,KHUKHRO20121}).

For bounding the nilpotency class of the group $G$, Lie ring methods are used. Indeed, the proof is based on the analogous results for Lie rings and algebras which are also important in their own right. The advantage of this method lies in the fact that it is usually easier to deal with Lie rings as they are more linear objects. On the other hand, both steps, from the group to the Lie ring and back, may be quite non-trivial.

Until now, the question whether the bound for the nilpotency class of $G$ could be made independent of the order of the Frobenius complement $H$ remained unsolved. Indeed, there were no explicit examples showing the opposite. In this paper, we shall give a negative answer to this question by proving the following

\begin{theorem}\label{thm1}
There exists a family $\mathfrak{G}$ of finite nilpotent groups, of unbounded nilpotency class, whose members $G$ satisfy the conditions:
\begin{enumerate}
\item $G$ admits a metacyclic Frobenius group of automorphisms; 
\item the centralizer of the kernel in $G$ is trivial;
\item the centralizer of the complement in $G$ is abelian. 
\end{enumerate}
\end{theorem}

This result is based on an analogous one for Lie algebras. In this case, the transition to finite groups is obtained from the Lazard correspondence.

The main part of this paper is devoted to proving the result for Lie algebras. This will consist of the explicit construction of a family $\mathfrak{L}$ of Lie algebras, satisfying analogous properties. As a preliminary step, for every prime number $p$ we will construct a Lie algebra $L_p$ which admits a metacyclic Frobenius group of automorphisms of the form $C_p \rtimes C_{p-1}$. Its definition and properties will be discussed in Section~\ref{2}. In particular, we shall see that every algebra $L_p$ is $\mathbb{Z}$-graded. In the same section, we will also calculate the exact dimension of its homogeneous components (Proposition \ref{l_p,k}).

Let us denote by $I_p$ and $J_p$ the smallest invariant ideals generated respectively by $C_{L_p}(C_p)$ and by the derived subalgebra of $C_{L_p}(C_{p-1})$. We will study the former in Section~\ref{3} and the latter in Section~\ref{4}. Since they are both homogeneous, we will study their intersections with the homogeneous components, bounding the dimension of the subspaces thus obtained (Propositions \ref{i_pk} and \ref{j_pk}).

In Section~\ref{5} we will build the family $\mathfrak{L}$, whose elements are the quotient algebras $L_p/(I_p+J_p)$. We will prove the existence, for any natural number $n$, of a prime $p$ such that the corresponding algebra $L_p/(I_p+J_p)$ in $\mathfrak{L}$ has nilpotency class at least $n$. For this part, the previous bounds in  Propositions \ref{l_p,k}, \ref{i_pk} and \ref{j_pk} will be required. We will finally discuss the conditions under which the Lazard correspondence can be applied.

\section{The family of Lie algebras $L_p$}\label{2}

A crucial step towards the proof of Theorem \ref{thm1} consists of the following
\begin{prop}\label{prop}
There exists a family $\mathfrak{L}$ of nilpotent Lie algebras, of unbounded nilpotency class, such that any member $L$ satisfies the conditions:
\begin{enumerate}
\item $L$ admits a metacyclic Frobenius group of automorphisms;
\item the centralizer of the kernel in $L$ is trivial;
\item the centralizer of the complement in $L$ is abelian.
\end{enumerate}
\end{prop}
It will be proved in Section \ref{5}. In the same section we will also see that, in order to produce the analogous result for groups, we will need to consider the case where the characteristic of the underlying field changes with $L$. This section instead is  devoted to the definition and the study of another  family of Lie algebras, which is instrumental in the construction of $\mathfrak{L}$. For every prime number $p$, we will define a $\mathbb{Z}$-graded Lie algebra $L_p$, as a quotient of a free metabelian Lie algebra. In Proposition \ref{l_p,k}, we will determine the exact dimension of each homogeneous component. Furthermore, we will construct a metacyclic Frobenius group acting by automorphisms on $L_p$. In this paper,  we use the notation $\langle U \rangle$ and $_{id}\langle U \rangle$ respectively for the Lie subalgebra and the ideal generated by a subset $U$.

Let $K$ be a field of characteristic coprime with $p$, containing a primitive $p$th root of unity. For generality, we will not specify its characteristic at this stage, because the whole construction works independently of it. However, we stress that, although this is not made explicit by notation, the characteristic may change together with the prime $p$. 

Let $X_p=\left\{ x_1, \ldots, x_{2p-2}\right\}$ be an ordered set and let $M(X_p)$ denote the free metabelian Lie algebra over $K$ with free generating set $X_p$. Following \cite{trove.nla.gov.au/work/12633643}, we know that $M(X_p)$ has a basis consisting of $X_p$ together with the left-normed Lie brackets of the form
\begin{equation}\label{BasisMet}
\left[x_{i_1}, x_{i_2},\ldots, x_{i_k} \right]\text{, with} \, k \ge 2, \; \text{and } x_{i_1}, \ldots, x_{i_k} \in X_p, \; x_{i_1}>x_{i_2} \le \ldots \le x_{i_k}.
\end{equation}
In particular, for every fixed number $k$, the above set represents a basis for the degree $k$ homogeneous component of $M(X_p)$, which is the $K$-vector space spanned by all Lie brackets of length $k$ involving the generators.

Renaming the generators in such a way that $X_p=\left\{a_1, \ldots, a_{p-1}, v_1, \ldots v_{p-1}\right\}$
and $a_1 <\ldots<a_{p-1}<v_1<\ldots<v_{p-1}$, we get that the set in (\ref{BasisMet}) consists of the following Lie brackets:
\begin{itemize}
\item $\left[a_{i_1},a_{i_2},\ldots, a_{i_j}, v_{i_{j+1}}, \ldots, v_{i_k} \right]$, for all $2 \le j \le k$, and $i_1>i_2 \le \ldots \le i_j$, $i_{j+1}\le \ldots \le i_{k}$,
\item $\left[v_{i_1},a_{i_2}, \ldots, a_{i_j},v_{i_{j+1}}, \ldots, v_{i_k} \right]$,  for all $2 \le j \le k$ and for all $i_2 \le \ldots \le i_j$, $i_{j+1}\le \ldots \le i_{k}$,
\item $\left[v_{i_1}, v_{i_2}, \ldots, v_{i_k} \right]$, for all $i_1>i_2 \le \ldots \le i_k$,
\end{itemize}
where $k \ge 2$ and $0 <i_1, i_2, \ldots ,i_k<p $.

Let $A$ denote the subalgebra generated by $a_1, \ldots, a_{p-1}$. Similarly, denote by $V$ the subalgebra generated by $v_1, \ldots, v_{p-1}$, and by $I_V$ the ideal generated by $V$.
Consider the quotient algebra $\tilde{L_p}=M(X_p)/[I_V,I_V]$. Since $I_V$ is a homogeneous ideal,  this Lie algebra inherits from $M(X_p)$ its $\mathbb{Z}$-grading. The generators form a basis for its degree $1$ homogeneous component. For every $k \ge 2$, a basis for its degree $k$ homogeneous component is given by the Lie brackets  
\begin{itemize}
\item $\left[a_{i_1},a_{i_2},\ldots, a_{i_k} \right]$, for all $i_1>i_2 \le \ldots \le i_k$,
\item $\left[a_{i_1},a_{i_2},\ldots, a_{i_{k-1}}, v_{i_k} \right]$, for all $i_1>i_2 \le \ldots \le i_{k-1}$,
\item $\left[v_{i_1},a_{i_2}, a_{i_3}\ldots, a_{i_k} \right]$, for all $i_2 \le \ldots \le i_k$,
\end{itemize}
where $0 <i_1, i_2, \ldots ,i_k<p $. Indeed, it is not hard to prove that $[I_V,I_V]$ is generated by the remaining Lie brackets, which are the ones with more than one entry in $V$.

Finally, define the Lie algebra $L_p$ to be $\tilde{L_p}/ _{id}\langle [A,A] \rangle$. Similarly, since the ideal generated by the derived algebra of $A$ is homogeneous, the algebra $L_p$ is $\mathbb{Z}$-graded. A basis for it consists of $X_p$ together with the Lie brackets
\begin{equation}\label{basis}
\mathcal{B}_{p,k}=\left\{[v_{i_0},a_{i_1}, \ldots, a_{i_{k-1}}] \, : \, 0<i_0<p, \, 0<i_1 \le i_2 \le \ldots \le i_{k-1}<p \right\},
\end{equation}
for every $k \ge 2$. By construction, it follows that the Lie algebra $L_p$ is a semidirect sum of the abelian ideal $I_V$ and the abelian subalgebra $A$. 

In the next proposition, we determine the dimension of the homogeneous components of $L_p$.
\begin{prop}\label{l_p,k}
For every prime $p$ and natural number $k$, let $l_{p,k}$ denote the dimension of $L_{p,k}$, the homogeneous component of $L_p$ of degree $k$. We have
\begin{displaymath}
l_{p,1}=2(p-1) \qquad  \text{and} \qquad l_{p,k}=(p-1)\binom{k+p-3}{k-1}\; \text{for all } k \ge 2.
\end{displaymath}
\end{prop}
\begin{proof}
Since the set $X_p$ is a basis for the subspace $L_{p,1}$, the first equality trivially holds. Now assume $k \ge 2$. By Formula (\ref{basis}), all possible Lie brackets in $\mathcal{B}_{p,k}$ with a fixed initial entry $v_{i_0}$ are determined by the $(k-1)$-combinations with repetition from $p-1$ elements.  The expression for $l_{p,k}$ follows from the fact that the number of these combinations is exactly  $\binom{(k-1)+(p-1)-1}{k-1}$ and $v_{i_0}$ varies among $p-1$ elements.
\end{proof}

The Lie algebra $L_p$ admits a Frobenius group of automorphisms of the form $C_p \rtimes C_{p-1}$, also denoted simply by $C_pC_{p-1}$. Let $f$ denote a generator of the Frobenius kernel $C_p$. For each $i=1, \ldots,p-1$, we define $a_i^f= \omega^i a_i$ and $v_i^f= \omega^i v_i$, where $\omega$ is a complex primitive $p$th root of unity. Let $h$ be a generator of the Frobenius complement $C_{p-1}$ and let $r$ be a primitive $(p-1)$th root of unity in $\mathbb{Z}_p$. We define $a_i^h=a_{ri}$ and $v_i^h=v_{ri}$, where the indices on the right-hand side must be taken modulo $p$. It is straightforward to check that these conditions determine a Frobenius group where $f^{h^{-1}}=f^r$.

Let $C_{L_p}(C_p)$ and $C_{L_p}(C_{p-1})$ respectively denote the centralizers of the Frobenius kernel and  the Frobenius complement. They are both $\mathbb{Z}$-graded subspaces, since the homogeneous components $L_{p,k}$ are $C_p C_{p-1}$-invariant. A basis for $C_{L_p}(C_{p}) \cap L_{p,k}$ is obtained by those Lie brackets 
$[v_{i_0},a_{i_1}, \ldots, a_{i_{k-1}}]$ in $\mathcal{B}_{p,k}$ satisfying the condition $i_0+i_1+ \cdots+ i_{k-1}\equiv 0\;  (\!\!\!\!\mod p)$.
The next proposition presents a basis for the subspace $C_{L_p}(C_{p-1}) \cap L_{p,k}$.

\begin{prop}\label{dimCentrCompl}
The basis $\mathcal{B}_{p,k}$ of $L_{p,k}$ is permuted by $C_{p-1}$. The number of its orbits is equal to the dimension of the subspace $C_{L_p}(C_{p-1}) \cap L_{p,k}$, namely $l_{p,k}/(p-1)$. Indeed, the set 
$\left\{ \sum_{j=0}^{p-2} b^{h^j}, \; b \in \mathcal{B}_{p,k}\right\}$
is a basis for such subspace.
\end{prop}
\begin{proof}
First we need to prove that the generator $h$ of $C_{p-1}$ acts as a permutation of $\mathcal{B}_{p,k}$. This follows directly from the formula for $\mathcal{B}_{p,k}$. Indeed, for every Lie bracket $[v_{i_0}, a_{i_1}, \ldots, a_{i_{k-1}}]$ in this set, its image under $h$ is $[v_{ri_0}, a_{ri_1}, \ldots, a_{ri_{k-1}}]$, again contained in $\mathcal{B}_{p,k}$, possibly after a rearrangement of the terms $a_{ri_j}$. Furthermore, every orbit has length $p-1$. Indeed, for $j=0, \ldots, p-2$, the images 
$[v_{r^ji_0}, a_{r^ji_1},\ldots,a_{r^ji_{k-1}}]$ under $h^j$ are all distinct because they differ in the first entry. As a consequence, we get that the number of orbits is $|\mathcal{B}_{p,k}|/(p-1)$.

To each orbit there corresponds one vector in $C_{L_p}(C_{p-1}) \cap L_{p,k}$, namely the sum over its elements. This follows directly from the structure of a permutation module, where to each cycle of the permutation there corresponds one eigenvector with eigenvalue $1$.
\end{proof}

The algebra $L_p$ also has a natural $\left(\mathbb{Z}/p \mathbb{Z}\right)$-grading, arising from the eigenspace decomposition corresponding to the action of $C_p$. Indeed, for every $j=0, \ldots, p-1$, define $^{j}\!L_p=\left\{x \in L_p | x^f=\omega^j x \right\}$. We call this subspace the $C_p$-homogeneous component of weight $j$. Because of the definition of $f$, a Lie bracket $[v_{i_0},a_{i_1}, \ldots, a_{i_{k-1}}]$ in $\mathcal{B}_{p,k}$ belongs to such subspace if and only if $i_0+i_1+\cdots+i_{k-1} \equiv j \, (\!\!\!\!\mod p)$. Clearly, the $C_p$-homogeneous component of weight $0$ corresponds to the centralizer $C_{L_p}(C_p)$ and it is $C_{p-1}$-invariant.

Since the homogeneous components $L_{p,k}$ are $C_p$-invariant, they are, in turn, $\mathbb{Z}_p$-graded by this basic fact of linear algebra:

\begin{lemma}\label{lemma}
Let $K$ be a field and $V$ a $K$-vector space. Let $\phi:V \rightarrow V$ be a linear map and $U$ a $\phi$-invariant subspace.  Assume that an element $u$ of $U$ can be written in the form
$u=u_1+\cdots+u_l$, where $u_i$ are eigenvectors of $\phi$ corresponding to distinct eigenvalues $\alpha_i$. Then $u_i \in U$ for all $1 \le i \le l$.
\end{lemma}

 Hence, we can write
\begin{equation*}
L_{p,k}= \bigoplus_{j=0}^{p-1} L_{p,k} \, \cap \,^{j}\!L_p.
\end{equation*}
The subspaces on the right-hand side are permuted by $h$, according to the formula $ (L_{p,k} \,\cap\, ^{j}\!L_p)^h= L_{p,k} \,\cap\, ^{rj}\!L_p$. Hence, for every prime $p$ and positive integer $k$, the subspaces $W_{p,k}=L_{p,k} \, \cap \,^{0}\!L_p$ and 
$W'_{p,k}=L_{p,k} \, \cap \,^{1}\!L_p \oplus \cdots \oplus L_{p,k} \, \cap \,^{p-1}\!L_p$ are $C_pC_{p-1}$-modules. The basis of $W'_{p,k}$ consists of eigenvectors for the action of $C_p$ which are permuted by $C_{p-1}$ in orbits of length $p-1$. Each of these orbits spans a $C_pC_{p-1}$-module of dimension $p-1$. 
Hence, we have that 
\begin{equation*}
\dim(W'_{p,k})=(p-1) \dim(C_{L_p}(C_{p-1}) \cap W'_{p,k}).
\end{equation*}
Using Proposition \ref{dimCentrCompl} and the $C_{p-1}$-invariance of both subspaces, we get the corresponding formula for $W_{p,k}$.

\section{The ideal generated by $C_{L_p}(C_p)$}\label{3}
Let $I_p$ denote the ideal generated by the centralizer $C_{L_p}(C_p)$. We have already mentioned that  $C_{L_p}(C_p)$ is degree-homogeneous and $C_pC_{p-1}$-invariant. Hence, the same properties hold for $I_p$. In this section, for every natural number $k$, we bound the dimension of the subspaces $I_p \cap L_{p,k}$. 

Since $v_i$ and $a_i$ are eigenvectors of $f$ with corresponding eigenvalues $\omega^i$, a basis for the subspace $I_p \cap L_{p,k}$ is given by those Lie brackets $[v_{i_0},a_{i_1}, \ldots, a_{i_{k-1}}]$ belonging to $\mathcal{B}_{p,k}$ for which there exists a permutation $\pi$ in the symmetric group $S_{k-1}$ such that $i_0+i_{\pi(1)}+\cdots+i_{\pi(s)} \equiv 0\; (\!\!\!\!\mod p)$, for some $s<k$.
In the next proposition, we bound the dimension of this subspace.

\begin{prop}\label{i_pk}
For every prime number $p$ and natural number $k$, let $i_{p,k}$ denote the dimension of the homogeneous subspace $I_p \cap L_{p,k}$. Then
\begin{equation*}
i_{p,1}=0 \qquad \text{and} \qquad   i_{p,k} \le (k-1)(p-1)^{k-1} \quad \text{for all } k \ge 2.
\end{equation*}
\end{prop}
\begin{proof}
First, we recall that $L_{p,1}$ is the subspace of homogeneous elements of $L_p$ of degree $1$ and it is generated by $X_p$. Since none of the $v_i, \, a_i$ is centralized by $C_p$, we clearly have $i_{p,1}=0$. More in general, we have
\begin{align*}
I_p \cap L_{p,k} &= \sum_{j=1}^k [C_{L_p}(C_p) \cap L_{p,j},\underbrace{L_{p,1},\ldots,L_{p,1}}_{k-j}]\\
           &=\sum_{j=1}^{k-1}[[C_{L_p}(C_p) \cap L_{p,j},\underbrace{L_{p,1},\ldots,L_{p,1}}_{k-1-j}],L_{p,1}]+C_{L_p}(C_p) \cap L_{p,k}\\
           &=\Big[\sum_{j=1}^{k-1}[C_{L_p}(C_p) \cap L_{p,j},\underbrace{L_{p,1},\ldots,L_{p,1}}_{k-1-j}],L_{p,1}\Big]+C_{L_p}(C_p) \cap L_{p,k}\\
           &=[I_p \cap L_{p,{k-1}},L_{p,1}]+C_{L_p}(C_p) \cap L_{p,k}.
\end{align*}
Hence, we obtain $i_{p,k} \le \dim([I_p \cap L_{p,{k-1}},L_{p,1}])+\dim(C_{L_p}(C_p) \cap L_{p,k})$. 

The subspace $C_{L_p}(C_p) \cap L_{p,k}$ has basis given by the Lie brackets $[v_{i_0},a_{i_1}, \ldots, a_{i_{k-1}}]$ in $\mathcal{B}_{p,k}$ for which $i_0+i_1+\cdots+ i_{k-1} \equiv 0\, (\!\!\!\!\mod p)$. This implies that the last index $i_{k-1}$ is uniquely determined by the previous ones and hence, $\dim(C_{L_p}(C_p) \cap L_{p,k})\le (p-1)^{k-1}$. Notice that this bound is only sharp for $k=2$, because the set $\{[v_i,a_{-i}], \;i=1, \ldots, p-1\}$ is a basis for $I_p \cap L_{p,2}$. However, for $k \ge 3$ this argument counts some basis elements in $\mathcal{B}_{p,k}$ several times, as we consider all simple Lie bracket of degree $k$, not only the ``ordered'' ones.

The subspace $[I_p \cap L_{p,{k-1}},L_{p,1}]$ is generated by the Lie brackets of the form $[x, a_j]$, where $x$ ranges over a basis of $I_p \cap L_{p,{k-1}}$ and $j=1, \ldots, p-1$. Hence, we have $\dim([I_p \cap L_{p,{k-1}},L_{p,1}]) \le (p-1) i_{p,k-1}$. Once again, we point out that the equality trivially holds for $k=2$ , and for $k=3$. The inequality is instead strict for larger values of $k$. These bounds together give the recursive formula
\begin{equation*}
i_{p,k} \le (p-1)i_{p,k-1}+(p-1)^{k-1},
\end{equation*}
from which the conclusion follows by induction.
\end{proof}

This proposition concludes this section and the study of the ideal $I_p$. Indeed, the bound on the dimension of its homogeneous components $I_p \cap L_{p,k}$ is all we need in the proof of Proposition \ref{prop}.

\section{The ideal generated by $[C_{L_p}(C_{p-1}),C_{L_p}(C_{p-1})]$}\label{4}
Let $J_p$ denote the smallest $C_pC_{p-1}$-invariant ideal containing the derived subalgebra $[C_{L_p}(C_{p-1}),C_{L_p}(C_{p-1})]$. Since $C_{L_p}(C_{p-1})$ is not $C_p$-invariant, then $J_p$ is strictly bigger than the ideal generated simply by $[C_{L_p}(C_{p-1}),C_{L_p}(C_{p-1})]$. Indeed, in this section  we will prove that $J_p$ is generated, as ideal, by the $C_p$-homogeneous components of the generators of $[C_{L_p}(C_{p-1}),C_{L_p}(C_{p-1})]$. This is done in Proposition \ref{Jp generators}. Moreover, as previously done for $I_p$, we will bound the dimension of the subspaces $J_p \cap L_{p,k}$, for every prime number $p$ and integer $k$ (Proposition \ref{j_pk}). The bound is obtained as the exact result of a recursion.

We start determining the generators of the subalgebra $[C_{L_p}(C_{p-1}),C_{L_p}(C_{p-1})]$. As already pointed out, the centralizer $C_{L_p}(C_{p-1})$ is degree-homogeneous and a basis for the subspaces 
$C_{L_p}(C_{p-1}) \cap L_{p,k}$ is given in Proposition \ref{dimCentrCompl}. The subalgebra $[C_{L_p}(C_{p-1}),C_{L_p}(C_{p-1})]$ is generated by Lie brackets of the form $[x,y]$, where $x$ and $y$ lie in the basis of $C_{L_p}(C_{p-1})$. Those Lie brackets are trivial unless one of the entries is $a_1+a_2+\cdots+a_{p-1}$. Hence, the derived subalgebra of $C_{L_p}(C_{p-1})$ is generated by the Lie brackets of the form $[x, a_1+\cdots+a_{p-1}]$, where $x$ lies in the basis of $C_{L_p}(C_{p-1})$.

In the next proposition we produce a set of generators of the ideal $J_p$.

\begin{prop} \label{Jp generators}
The smallest ideal containing $[C_{L_p}(C_{p-1}),C_{L_p}(C_{p-1})]$ and invariant under the action of the group $C_pC_{p-1} $ is generated, as an ideal, by the $C_p$-homogeneous components of the generators of $[C_{L_p}(C_{p-1}),C_{L_p}(C_{p-1})]$.
\end{prop}
\begin{proof}
Consider $[x, \sum_{i=1}^{p-1} a_i]$, where $x$ belongs to the basis of $C_{L_p}(C_{p-1})$ given in Proposition \ref{dimCentrCompl}. We can write $x$ as a linear combination of eigenvectors of $f$, namely $x=\sum_{j=0}^{p-1}x_j$, where $x_j$ is in the $C_p$-homogeneous component of weight $j$. Since $x$ is in $C_{L_p}(C_{p-1})$, we necessarily have that $x_{rj}=(x_j)^h$, for all $j=0,\ldots p-1$. The same relation is true for the $C_p$-homogeneous components of $[x, \sum_{i=1}^{p-1} a_i]$. Indeed, plugging the decomposition of $x$ into the Lie bracket, we can rewrite it as
$\sum_{j=0}^{p-1} \left( \sum_{k \ne j} [x_k, a_{j-k}]\right)$, where the sum in parenthesis represents its $C_p$-homogeneous component of weight $j$. Applying $h$ to it, we get
\begin{equation*}
\left( \sum_{k \ne j} [x_k, a_{j-k}]\right)^h=\sum_{k \ne j} [x_k^h, a_{j-k}^h]=
 \sum_{k \ne j}[x_{rk}, a_{rj-rk}],
\end{equation*}
where this should be interpreted as a sum over $k$, where $k$ varies from $0$ to $p-1$.
This is exactly the homogeneous component of weight $rj$. 

We have just shown that the set of the $C_p$-homogeneous components of the generators of $[C_{L_p}(C_{p-1}),C_{L_p}(C_{p-1})]$ is permuted by $C_{p-1}$. This implies that the ideal generated by this set is $C_{p-1}$-invariant. Moreover, this ideal is also $C_p$-invariant because, by construction, its generators are $C_p$-homogeneous. It contains the subalgebra $[C_{L_p}(C_{p-1}),C_{L_p}(C_{p-1})]$ because it contains all its generators (as a sum of their $C_p$-homogeneous components). Hence, it contains $J_p$, for minimality of $J_p$.

To prove the opposite inclusion, it is sufficient to show that all $C_p$-homogeneous components of the generators of $[C_{L_p}(C_{p-1}),C_{L_p}(C_{p-1})]$ are in $J_p$. This is true by Lemma \ref{lemma}.
\end{proof}
We can now bound the dimension of the degree-homogeneous components of $J_p$.
\begin{prop}\label{j_pk}
For every prime number $p$ and natural number $k$, let $j_{p,k}$ denote the dimension of the subspace $J_p \cap L_{p,k}$, the homogeneous component of $J_p$ of degree $k$. Then
\begin{equation*}
j_{p,1}=0 \qquad \text{and} \qquad j_{p,k} \le p\left(\sum_{j=0}^{k-2} (p-1)^{k-2-j} \binom{j+p-2}{j}\right) \, \text{for all } k \ge 2.
\end{equation*}
\end{prop}
\begin{proof}
First, observe that the ideal $J_p$ is contained in the derived subalgebra  of $L_p$. This implies trivial intersection with the subspace $L_{p,1}$ of homogeneous elements of degree $1$ and hence $j_{p,1}=0$.

More generally, if $x$ ranges in the set of generators of $[C_{L_p}(C_{p-1}),C_{L_p}(C_{p-1})]$ and $x_l$ denotes its $C_p$-homogeneous component of weight $l$, by the previous result we can write
\begin{align*}
J_p \cap L_{p,k}&\! =\!\sum_{j=1}^k \! \Big(\! \sum_{l=0}^{p-1} \big[\langle x_l : x \in [C_{L_p}(C_{p-1}),C_{L_p}(C_{p-1})] \cap L_{p,j} \rangle,\underbrace{L_{p,1},\ldots,L_{p,1}}_{k-j}\big]\Big)\\
           &\!=\!\sum_{j=1}^{k-1} \! \Big(\! \sum_{l=0}^{p-1} \big[\big[\langle x_l :  x \in [C_{L_p}(C_{p-1}),C_{L_p}(C_{p-1})] \cap L_{p,j} \rangle,\underbrace{L_{p,1},\ldots,L_{p,1}}_{k-1-j}\big],L_{p,1}\big]\Big)\\
           &\quad + \sum_{l=0}^{p-1} \langle x_l :  x \in [C_{L_p}(C_{p-1}),C_{L_p}(C_{p-1})] \cap L_{p,k} \rangle\\
           &=[J_p \cap L_{p,k-1},L_{p,1}] + \sum_{l=0}^{p-1} \langle x_l : x \in [C_{L_p}(C_{p-1}),C_{L_p}(C_{p-1})] \cap L_{p,k} \rangle.
\end{align*}
Hence, we have $j_{p,k} \le \dim([J_p \cap L_{p,k-1},L_{p,1}])+p \dim([C_{L_p}(C_{p-1}),C_{L_p}(C_{p-1})] \cap L_{p,k})$. 

The subalgebra $[J_p \cap L_{p,k-1},L_{p,1}]$, which is non-trivial only when $k \ge 3$, is generated by the Lie brackets of the form $[x,a_j]$, where $x$ is an element of the basis of $J_p \cap L_{p,k-1}$ and $j=1, \ldots, p-1$. Indeed, the Lie brackets $[x, v_i]$ are trivial by construction. Hence, we get 
$\dim([J_p \cap L_{p,k-1},L_{p,1}]) \le (p-1) j_{p,{k-1}}$.

As previously remarked, the subalgebra $[C_{L_p}(C_{p-1}),C_{L_p}(C_{p-1})] \cap L_{p,k}$ is generated by the elements $[x, \sum_{i=1}^{p-1} a_i]$, where $x$ is in the basis of $C_{L_p}(C_{p-1}) \cap L_{p,{k-1}}$. Hence, its dimension is not greater than the dimension of $C_{L_p}(C_{p-1}) \cap L_{p,{k-1}}$, namely $l_{p,k-1}/(p-1)$, by Proposition \ref{dimCentrCompl}. Using the expression for $l_{p,k}$ computed in Proposition \ref{l_p,k}, we get the following recursive formula
\begin{equation*}
j_{p,k} \le (p-1)j_{p,k-1}+p \binom{k+p-4}{k-2}.
\end{equation*}
We now prove the desired bound for $j_{p,k}$ by induction on $k$. Observe that the subspace $[C_{L_p}(C_{p-1}),C_{L_p}(C_{p-1})] \cap L_{p,2}$ is one-dimensional, generated by the Lie bracket $[\sum_{i=1}^{p-1}v_i, \sum_{j=1}^{p-1}a_j]$. According to Proposition \ref{Jp generators}, its $C_p$-homogeneous components generate $J_p \cap L_{p,2}$. Hence, the bound holds for $k=2$. Suppose now that the bound is true for $k-1$. Using the above inequality in the inductive hypothesis, we find:
\begin{align*}
j_{p,k} &\le (p-1)p \left( \sum_{j=0}^{k-3} (p-1)^{k-3-j} \binom{j+p-2}{j} \right)+p \binom{k+p-4}{k-2}\\
&= p \left(\sum_{j=0}^{k-3} (p-1)^{k-2-j} \binom{j+p-2}{j} \right)+p \binom{k+p-4}{k-2}\\
& = p \left(\sum_{j=0}^{k-2} (p-1)^{k-2-j} \binom{j+p-2}{j} \right).
\end{align*}
This concludes the proof.
\end{proof}

This also concludes the section about the ideal $J_p$. As in the case of the ideal $I_p$, all we will need later on is the bound for the dimension of its homogeneous components. 

\section{Proofs of Proposition \ref{prop} and Theorem \ref{thm1}} \label{5}
In this last section, we will prove  Proposition \ref{prop} by explicit construction of the Lie algebras $\bar{L}_p$ belonging to the family $\mathfrak{L}$. We will then apply Lazard correspondence to prove Theorem \ref{thm1}, where we will need to impose some conditions on the underlying field of each $\bar{L}_p$ in $\mathfrak{L}$. Recall that, until now, the only requirements on the underlying field of $L_p$ were having characteristic different from $p$ and containing a primitive $p$th root of unity.

\begin{proof}[Proof of Proposition \ref{prop}]For every prime number $p$, define $\bar{L}_p$ to be the quotient algebra $L_p/(I_p+J_p)$. Since the two ideals $I_p$ and $J_p$ are graded, the same is true for their sum. In particular, every homogeneous component is given by the sum of the corresponding ones: $(I_p+J_p) \cap L_{p,k}=I_p \cap L_{p,k}+J_p \cap L_{p,k}$. This implies that $\bar{L}_p$ inherits from $L_p$ its $\mathbb{Z}$-grading . Let $\bar{L}_{p,k}$ denote the image of $L_{p,k}$ in the quotient and let $\bar{l}_{p,k}$ indicate its dimension. We have
\begin{align*}
\bar{l}_{p,k} &= l_{p,k}- \dim((I_p+J_p) \cap L_{p,k})=l_{p,k}- \dim(I_p \cap L_{p,k} +J_p \cap L_{p,k})\\
              & \ge l_{p,k}- \dim(I_p \cap L_{p,k})-\dim(J_p \cap L_{p,k})= l_{p,k}-i_{p,k}-j_{p,k}.
\end{align*}

Substituting in this formula the expression for $l_{p,k}$ from Proposition \ref{l_p,k}, and the upper bounds for $i_{p,k}$ and $j_{p,k}$, respectively from Proposition \ref{i_pk} and \ref{j_pk}, we get
\begin{equation*}
\bar{l}_{p,k} \ge (p-1)\binom{k+p-3}{k-1}-(k-1)(p-1)^{k-1}-p\left(\sum_{j=0}^{k-2} (p-1)^{k-2-j} \binom{j+p-2}{j}\right).
\end{equation*}
We observe that the right-hand side is a polynomial function of the prime number $p$, whose leading term is $p^k/(p-1)!$ and comes from the formula for $l_{p,k}$.

The algebra $\bar{L}_p$ admits the same Frobenius group $C_pC_{p-1}$, as a group of automorphisms. Indeed, since $I_p$ and $J_p$ are $C_p C_{p-1}$-invariant, the generator $f$ of $C_p$ and the generator $h$ of $C_{p-1}$ induce automorphisms of the quotient algebra, defined in the canonical way. The orders of $f$ and $h$ remain unchanged in their action on the quotient, namely  equal to $p$ and $p-1$ respectively.

The fixed-point subspace $C_{\bar{L}_p}(C_p)$ is trivial. Indeed, in case of a coprime action, the fixed points in the quotient algebra are covered by the fixed points in the original Lie algebra (see for example \cite[Theorem~1.6.2]{khukhro1993nilpotent}). For the same reason, if we require that, for every prime $p$, the characteristic of the underlying field of $\bar{L}_p$ is coprime with $p-1$, then the subspace $[C_{\bar{L}_p}(C_{p-1}),C_{\bar{L}_p}(C_{p-1}) ]$ is trivial too. At the end of this section we will show that this hypothesis is unnecessary. Indeed, the fixed points of $C_{p-1}$ in the quotient algebra are covered by the fixed points in the original Lie algebra even when the characteristic of the field divides $p-1$. Nevertheless, for now just assume that this condition holds, so that the three properties of $\mathfrak{L}$ listed in Proposition \ref{prop} are satisfied and we can continue with our proof.

It only remains to show that for every integer $k$, there exists a Lie algebra in $\mathfrak{L}$ with nilpotency class at least $k$. In other words, this algebra must have a non-trivial Lie bracket of degree $k$. Because the lower bound for $\bar{l}_{p,k}$ has positive leading coefficient, for every fixed $k$ we can always find a prime number $p$ for which $\bar{l}_{p,k}$ is positive. Hence, the corresponding Lie algebra will have nilpotency class at least $k$. Proposition \ref{prop} is now proved modulo Lemma \ref{noCoprim}.
\end{proof}

\begin{proof}[Proof of Theorem \ref{thm1}]Every $\bar{L}_p$ is a metabelian $\left(\mathbb{Z}/p \mathbb{Z} \right)$-graded Lie algebras with $^{0}\bar{L}_p=0$. According to a theorem of Shumyatsky \cite[Theorem 3.3]{1627225420050301}, its nilpotency class is at most $p-1$. We can therefore produce a similar result for groups under the extra condition that, for every algebra $\bar{L}_p$ (and hence $L_p$), the characteristic $q_p$ of the  underlying field is greater than $p$. Indeed, under this hypothesis, the Lazard correspondence based on the `truncated' Baker--Campbell--Hausdorff formula (see, for example, \cite[Section 10.2]{9780511526008}) transforms every $\bar{L}_p$ into a finite $q_p$-group $Q_p$ with the same nilpotency class as $\bar{L}_{p}$. The group $Q_p$ admits the same Frobenius group of automorphisms $C_pC_{p-1}$ with trivial $C_{Q_p}(C_p)$ and abelian $C_{Q_p}(C_{p-1})$. This proves Theorem \ref{thm1}.
\end{proof}

Increasing $q$ with $p$ in our construction was only required for an application of the Lazard correspondence. Hence, it is not unnatural to ask the following
\begin{question}
We wonder if it is possible to construct a family of $q$-groups, of unbounded nilpotency class, where $q$ is a fixed prime number, whose elements $G$ satisfy:
\begin{enumerate}
\item $G$ admits a metacyclic Frobenius group of automorphisms;
\item the centralizer of the Frobenius kernel in $G$ is trivial;
\item the centralizer of the Frobenius complement in $G$ is abelian.
\end{enumerate}
\end{question}

We conclude this section by proving that the fixed points of $C_{p-1}$ in $\bar{L}_p$ are covered by the fixed points of the same group in $L_p$ even when the action is not coprime.
 \begin{lemma}\label{noCoprim}
 For every prime number $p$ let $L_p$, $\bar{L}_p$, $I_p$ and $J_p$ defined as above. Let $T_p$ denote the ideal $I_p+J_p$, so that $\bar{L}_p=L_p/T_p$. Then we have
 \begin{equation*}
 C_{\bar{L}_p} (C_{p-1})=(C_{L_p} (C_{p-1})+ T_p)/ T_p.
 \end{equation*} 
 \end{lemma}
\begin{proof} We can restrict ourselves to the degree-homogeneous case and, denoting the subspace $(I_p+J_p) \cap L_{p,k}$ by $T_{p,k}$ , show that
\begin{equation*}
C_{L_{p,k}/T_{p,k}}(C_{p-1})=(C_{L_{p,k}}(C_{p-1})+T_{p,k})/ T_{p,k}.
\end{equation*}

Recall that $W_{p,k}=L_{p,k} \, \cap \,^{0}\!L_p$ and 
$W'_{p,k}=L_{p,k} \, \cap \,^{1}\!L_p \oplus \cdots \oplus L_{p,k} \, \cap \,^{p-1}\!L_p$.
Because $I_p$ and $J_p$ are $C_p$-invariant we have
\begin{equation*}
T_{p,k}= (I_p+J_p) \cap W_{p,k} \oplus (I_p+J_p) \cap W'_{p,k}=W_{p,k}\oplus (I_p+J_p) \cap W'_{p,k},
\end{equation*}
whence
\begin{equation*}
\dfrac{L_{p,k}}{T_{p,k}}
\cong \dfrac{W'_{p,k}}{(I_p+J_p) \cap W'_{p,k}}.
\end{equation*}

 A set of generators for $(I_p+J_p) \cap W'_{p,k}$ is given by the generators of $I_p\cap W'_{p,k}$ together with the generators of $J_p \cap W'_{p,k}$. In other words, since the generators of $(I_p+J_p) \cap L_{p,k}$, described in Propositions \ref{i_pk}  and \ref{Jp generators}, are all $C_p$-homogeneous, we take those of them not belonging to $^{0}\!L_p$.

It is not hard to show that the generators of $I_p\cap W'_{p,k}$ are permuted by $C_{p-1}$ in orbits of length $p-1$, whose elements are all in different $C_p$-homogeneous components (one for each $^{j}\!L_p$, with $1 \le j \le p-1$). The same holds for the generators of $J_p \cap W'_{p,k}$. The key point here is that each simple Lie bracket involved in one of those generators has sum of its indices different from $0$ modulo $p$. We have just shown that
\begin{equation*}
(I_p+J_p) \cap W'_{p,k}=\bigoplus_{j=0}^{p-2} ((I_p+J_p)\cap L_{p,k}\cap \,^{1}\!L_p)^{h^{j}}.
\end{equation*}

Now consider a basis $\left\{w_1,\ldots, w_t \right\}$ for the subspace $(I_p+J_p) \cap L_{p,k} \, \cap \,^{1}\!L_p $. From the above equation it follows that $(I_p+J_p) \cap W'_{p,k}$
has basis 
\begin{equation*}
\left\{w_1, \ldots, w_t, w_1^h, \ldots, w_t^h, \ldots, w_1^{h^{p-2}}, \ldots ,w_t^{h^{p-2}} \right\}.
\end{equation*}
If we complete the basis of $(I_p+J_p)\cap L_{p,k}\cap \,^{1}\!L_p$ to a basis of $L_{p,k} \, \cap \,^{1}\!L_p$,  by adding the vectors $\left\{w_{t+1}, \ldots ,w_s \right\}$, then a basis of $L_{p,k}/(I_p+J_p) \cap L_{p,k}$ is 
\begin{equation*}
\left\{w_{t+1}, \ldots, w_s, w_{t+1}^h, \ldots, w_{s}^h, \ldots, w_{t+1}^{h^{p-2}}, \ldots, w_{s}^{h^{p-2}} \right\},
\end{equation*}
where, with a little abuse of notation, the cosets are named by their representative.
The desired conclusion follows.
\end{proof}
We remark that the result of the previous lemma strongly relies on the definitions of the algebras involved. Indeed, a crucial step in the proof is the fact that both $L_p$ and $\bar{L}_p$ are free $H$-modules.

\bibliography{bibliography}

\end{document}